\documentclass[a4paper, reqno]{amsart}
\usepackage{graphicx}
\usepackage[bookmarksnumbered,colorlinks,plainpages]{hyperref}

\usepackage{float}
\usepackage{amsthm}
\usepackage[all,2cell]{xy}
\usepackage{amsfonts}
\usepackage{amsmath}
\usepackage{amssymb}
\usepackage{xypic,leqno,amscd,amssymb,pstricks,latexsym,amsbsy,xypic,mathrsfs,verbatim}
\usepackage{multirow}
\usepackage{booktabs}
\usepackage{makecell}
\usepackage{cases}
\usepackage{appendix}
\usepackage{graphicx,colortbl}

\usepackage{etex}
\usepackage{mathtools}

\usepackage{tikz}
\usetikzlibrary{positioning,arrows}
\usetikzlibrary{decorations.pathreplacing}

\UseAllTwocells \SilentMatrices
\newtheorem{thm}{Theorem}[section]

\newtheorem{cor}[thm]{Corollary}
\newtheorem{lem}[thm]{Lemma}

\newtheorem{prop}[thm]{Proposition}
\theoremstyle{definition}

\newtheorem{exm}[thm]{Example}
\theoremstyle{remark}

\numberwithin{equation}{section}

\newtheorem{innercustomthm}{{\bf Theorem}}
\newenvironment{customthm}[1]
{\renewcommand\theinnercustomthm{#1}\innercustomthm}
{\endinnercustomthm}

\newtheorem{innercustomques}{{\bf Question}}
\newenvironment{customques}[1]
{\renewcommand\theinnercustomques{#1}\innercustomques}
{\endinnercustomques}

\begin{document}
\input xy
\xyoption{all}

\title[On two open questions for extension bundles]
{On two open questions for extension bundles}

\author[Qiang Dong]{Qiang Dong}
\address{School of Mathematical Sciences, Shanghai Jiao Tong University, Shanghai, 200240, P.R. China.}
\email{dongqiang\_sjtu@sjtu.edu.cn}

\author[Shiquan Ruan]{Shiquan Ruan$^*$}
\address{School of Mathematical Sciences, Xiamen University, Xiamen 361005, P.R.China}
\email{sqruan@xmu.edu.cn}

\makeatletter \@namedef{subjclassname@2020}{\textup{2020} Mathematics Subject Classification} \makeatother

\thanks{$^*$ the corresponding author}
\date{\today}
\maketitle

\begin{abstract}
In this paper we give positive answers for two open questions on extension bundles over weighted projective lines, raised by Kussin, Lenzing and Meltzer in the paper ``Triangle singularities, ADE-chains and weighted projective lines''.
\end{abstract}

\section{Introduction}

Weighted projective lines and their coherent sheaf categories are introduced by Geigle and Lenzing in \cite{[GL]}, in order to give a geometric realization of canonical algebras in the sense of Ringel \cite{[Rin]}.
The subcategory $\textup{vect}\mbox{-}\mathbb{X}$ of vector bundles over a weighted projective line $\mathbb{X}$ carries a distinguished Frobenius exact structure, with the system of all line bundles as the indecomposable projective-injective objects, and then the induced stable category $\underline{\textup{vect}}\mbox{-}\mathbb{X}$ is a triangulated category by a general result of Happel \cite{[H]}. The stable category of vector bundles are closely related to many mathematical subjects, such as Kleinian singularity and Fuchsian singularity \cite{Lenzing2011Weighted,[ld],[E2007], [G1974]}, matrix factorization and Cohen-Macaulay module \cite{[KST1],[KST2]}, submodule category and monomorphism category \cite{[Chen1],[Chen],GaoZhang, [HLXZ],LuoZhang,MarPena, RS2006,RS2008,[RS], [RZ],[RZ2017],XiongZhangZhang,[Z]}.

In \cite{[KLM2]}, Kussin, Lenzing and Meltzer established a surprising link between
vector bundles over weighted projective lines and the invariant subspace problem of nilpotent operators as studied by Ringel and Schmidmeier \cite{[RS]}, a problem with a longstanding history going back to Birkhoff \cite{[Birk]}. They showed that
the stable category of invariant subspaces of nilpotent operators of nilpotency degree $p$ is equivalent to the stable category $\underline{\textup{vect}}\mbox{-}\mathbb{X}$ for $\mathbb{X}$ of weight type $(2,3,p)$ as triangulated categories. Consequently, these categories have Serre duality and admit tilting objects, whose endomorphism algebras are Nakayama algebras.

The results in \cite{[KLM2]} have many generalizations in different directions. Ladkani \cite{[Lad]} showed that the above Nakayama algebras (line shape) and tensor algebras (rectangle shape) are derived equivalent to certain Auslander algebras (triangle shape), which have been applied in \cite{[LMR]} to classify Nakayama algebras which are derived equivalent to Fuchsian singularities.
Simson considered Birkhoff type problems for nilpotent linear operators for ``two-flag one-peak'' posets in \cite{[Simson2015],[Simson2018]}, and related them to weighted projective line of (general) weight type $(p_1,p_2,p_3)$. For recollements and ladders associated to the stable categories of vector bundles, we refer to \cite{[Chen2012],[LL], [R]}.

For any weighted projective line $\mathbb{X}$ of weight type ${\mathbf{p}}=(p_1,p_2,p_3)$, the Picard group $\textup{Pic}(\mathbb{X})$ of $\mathbb{X}$ is isomorphic to the rank one abelian group $\mathbb{L}=\mathbb{L}({\mathbf{p}})$ on three generators $\vec{x}_1,\vec{x}_2,\vec{x}_3$ subject to the relations $p_1\vec{x}_1=p_2\vec{x}_2=p_3\vec{x}_3:=\vec{c}$, where $\vec{c}$ is called the canonical element of $\mathbb{L}$. For any line bundle $L$ and $\vec{x}$ in $\mathbb{L}$ with $0\leq\vec{x}\leq \sum\limits_{i=1}^{3}(p_i-2)\vec{x}_i$, the indecomposable middle term of the following ``unique'' non-split exact sequence
\begin{equation}\label{extension term}0\rightarrow L(\vec{\omega})\rightarrow  E_L\langle\vec{x}\rangle\to L(\vec{x})\rightarrow  0\end{equation} is called the \emph{extension bundle}, where $\vec{\omega}=\vec{c}-\sum\limits_{i=1}^{3}\vec{x}_i$ is the dualizing element of $\mathbb{L}$.
In particular, it is called an \emph{Auslander bundle} if $\vec{x}=0$.

It has been proven that each indecomposable rank-two bundle in $\textup{coh}\mbox{-}\mathbb{X}$ is an extension bundle, which plays a key role in investigating $\underline{\textup{vect}}\mbox{-}\mathbb{X}$, since the coherent sheaves of rank-zero (torsion sheaves) and of rank-one (line bundles) vanished in $\underline{\textup{vect}}\mbox{-}\mathbb{X}$. Many properties of this triangulated category were discussed in \cite{[KLM]}, and almost all the major results there are based on the features of extension bundles.

Denote by $\mathcal{V}_2$ the set of isomorphism classes of extension bundles in $\textup{vect}\mbox{-}\mathbb{X}$. Then $\textup{Pic}(\mathbb{X})$ acts on $\mathcal{V}_2$ by line bundle twist
\begin{align*}\textup{Pic}(\mathbb{X})\times\mathcal{V}_2\rightarrow\mathcal{V}_2;\quad (\vec{x},E)\mapsto E(\vec{x}).
\end{align*}
While the action is transitive for Auslander bundles, it is not the case for extension bundles. Kussin, Lenzing and Meltzer in \cite{[KLM]} have showed that the number of the orbit set $\mathcal{V}_2/\textup{Pic}(\mathbb{X})$ has an upper bound $(p_1-1)(p_2-1)(p_3-1)$. There is an open question stated in \cite[Remark 9.4]{[KLM]}:

\begin{customques}{\bf A}(\cite[Remark 9.4]{[KLM]})\label{AA}
What is the exact number of the orbit set $\mathcal{V}_2/\textup{Pic}(\mathbb{X})$?
\end{customques}

In this paper, we describe when two extension bundles are in the same orbit under the line bundle twist action in Section \ref{section2}. Basing on this observation, we obtain the exact number $|\mathcal{V}_2/\textup{Pic}(\mathbb{X})|$ via a Klein four-group action on certain set of extension bundles.

\begin{customthm}{\bf A}
[Theorem~\ref{exact number of orbits}] \label{Thm A}
Let $m$ be the number of even weights in the weight sequence $(p_1, p_2, p_3)$. Then
$$|\mathcal{V}_2/\textup{Pic}(\mathbb{X})|=\left\{
\begin{array}{ll}
\frac{1}{4}\prod\limits_{i=1}^{3}(p_i-1) & m\leq 1\\
\frac{1}{4}(\prod\limits_{i=1}^{3}(p_i-1)+(p_j-1)) & m=2\textup{\;\;with\;\;}p_j\textup{\;odd}\\
\frac{1}{4}(\prod\limits_{i=1}^{3}(p_i-1)+\sum\limits_{i=1}^{3}(p_i-1)) & m=3.
\end{array}\right.$$
\end{customthm}

For weight type $(2,3,p),\;p\geq2$, Kussin, Lenzing and Meltzer have constructed a tilting object in $\underline{\textup{vect}}\mbox{-}\mathbb{X}$ consisting of Auslander bundles, see \cite[Proposition 6.9]{[KLM]}. They raise the following open question.

\begin{customques}{\bf B}(\cite[Remark 6.10]{[KLM]})\label{bB}
For weight type $(2,a,b)$ with $a,b\geq2$, does there exist a tilting object in $\underline{\textup{vect}}\mbox{-}\mathbb{X}$ consisting of Auslander bundles?
\end{customques}

The second main result of this paper is to give a confirmation of this question.

\begin{customthm}{\bf B}
[Theorems~\ref{tilting auslander bundle 1} and \ref{tilting auslander bundle 2}]\label{Thm B}
Let $\mathbb{X}$ be a weighted projective line of weight type $(2,a,b)$ with $a,b\geq 2$, and let $E$ be an Auslander bundle. Then there are two tilting objects in $\underline{\textup{vect}}\mbox{-}\mathbb{X}$ as follows
\begin{itemize}
\item[(1)] $T_1=\bigoplus\{E(i\bar{x}_2+j\bar{x}_3)|0\leq i\leq b-2, 0\leq j\leq a-2\}$;
\item[(2)] $T_2=\bigoplus\{E(i\bar{x}_1+j\bar{x}_3)|0\leq i\leq b-2,0\leq j\leq a-2\}$.
\end{itemize}
\end{customthm}

This paper is organized as follows. In Section \ref{section2}, we describe when two extension bundles are in the same orbit and prove Theorem \ref{Thm A}. In Section \ref{SectionApp},
we show that projective covers and injective hulls are invariants for extension bundles, and investigate the stability for them. Section \ref{Section4} is devoted to proving Theorem \ref{Thm B}.

\section{Orbit counting formula for extension bundles}\label{section2}
Let ${\bf{k}}$ be an algebraically closed field and $\mathbb{X}$ be a weighted projective line of weight type $(p_1, p_2, p_3)$, where the integers $p_i\geq 2$ for $i=1,2,3$. Following \cite{[GL]}, we denote by $\textup{coh}\mbox{-}\mathbb{X}$ the category of coherent sheaves on $\mathbb{X}$ and $\textup{vect}\mbox{-}\mathbb{X}$ its full subcategory consisting of vector bundles.

\subsection{A canonical basis of the Grothendieck group}
Recall that each element in $\mathbb{L}:=\mathbb{L}(p_1,p_2,p_3)$ can be uniquely written in normal form
\begin{equation}\label{normal form}
\vec{x}=\sum\limits_{i=1}^{3}l_i\vec{x}_i+l\vec{c}, \textup{\ \ with\ \ }l_i,
l\in\mathbb{Z}\textup{\ \ and\ \ }0\leq l_i\leq p_i-1.
\end{equation}
Up to isomorphisms the line bundles are given by the system $\mathscr{L}$ of twisted structure sheaves $\mathcal{O}(\vec{x})$ with $\vec{x}\in\mathbb{L}$, where $\mathcal{O}$ is the structure sheaf of $\mathbb{X}$.

Denote by $K_0(\mathbb{X})$ the Grothendieck group of $\textup{coh}\mbox{-}\mathbb{X}$. By \cite[Proposition 4.1]{[GL]},
$T_{\textup{can}}=\bigoplus_{0\leq\vec{x}\leq\vec{c}}\mathcal{O}(\vec{x})$ is a canonical tilting sheaf in $\textup{coh}\mbox{-}\mathbb{X}$. Hence their classes $[\mathcal{O}(\vec{x})]$ with $0\leq \vec{x}\leq \vec{c}$ form a basis of $K_0(\mathbb{X})$. Therefore, each element in the Grothendieck group can be expressed as a linear combination of this basis. In particular, for line bundles we have the following explicit expressions:

\begin{prop}\label{Grothendieck group of line bundles in normal form}
Assume that $\vec{x}=\sum\limits_{i=1}^{3}l_{i}\vec{x}_{i}+l\vec{c}$ is in
normal form. Then we have
\begin{equation}\label{2.1}[\mathcal{O}(\vec{x})]=\sum\limits_{i=1}^{3}[\mathcal{O}(l_{i}\vec{x}_{i})]+l[\mathcal{O}(\vec{c})] - (l+2)[\mathcal{O}].\end{equation}
\end{prop}

\begin{proof}
Let $S_{i,j+1}$ be the exceptional simple sheaf arising from the exact sequence:
$$0\rightarrow\mathcal{O}(j\vec{x}_i)\rightarrow\mathcal{O}\big((j+1)\vec{x}_i\big)\rightarrow S_{i,j+1}\rightarrow0$$
for $i=1,2,3$ and $j=0,1,\cdots,p_i-1$. Note that $S_{i,j}(\vec{y})=S_{i,j+l'_i}$ for any $\vec{y}=\sum\limits_{i=1}^{3}l'_{i}\vec{x}_{i}+l'\vec{c}$. Then
$$[\mathcal{O}(\vec{x})]-[\mathcal{O}(\vec{x}-l_{i}\vec{x}_{i})]=\sum\limits_{j=1}^{l_i}[S_{i,j}]=[\mathcal{O}(l_{i}\vec{x}_{i})]-[\mathcal{O}].$$
Therefore, \begin{equation}\label{OS1}[\mathcal{O}(\vec{x})]-[\mathcal{O}(l\vec{c})]=
[\mathcal{O}(\vec{x})]-[\mathcal{O}(\vec{x}-\sum\limits_{i=1}^{3}l_i\vec{x}_i)]=
\sum\limits_{i=1}^{3}\big([\mathcal{O}(l_{i}\vec{x}_{i})]-[\mathcal{O}]\big).\end{equation} Let $S$ be an ordinary simple sheaf. Then $S(\vec{y})=S$ for any $\vec{y}\in\mathbb{L}$, and we have $[\mathcal{O}(l\vec{c})]-[\mathcal{O}\big((l-1)\vec{c}\big)]=[S]=[\mathcal{O}(\vec{c})]-[\mathcal{O}]$ for any $l\in\mathbb{Z}$. Hence, \begin{equation}\label{OS2}[\mathcal{O}(l\vec{c})]=l([\mathcal{O}(\vec{c})]-[\mathcal{O}])+[\mathcal{O}].\end{equation}
Combining (\ref{OS1}) and (\ref{OS2}) we have
$$[\mathcal{O}(\vec{x})]=([\mathcal{O}(\vec{x})]-[\mathcal{O}(l\vec{c})])+[\mathcal{O}(l\vec{c})]=\sum\limits_{i=1}^{3}[\mathcal{O}(l_{i}\vec{x}_{i})]+l[\mathcal{O}(\vec{c})] - (l+2)[\mathcal{O}].$$
This finishes the proof.
\end{proof}

Recall that the rank function $\textup{rk}$ is a linear form on $K_0(\mathbb{X})$ characterized by $\textup{rk}(L)=1$ for each $L\in\mathscr{L}$. The determinant function $\textup{det}:\;K_0(\mathbb{X})\rightarrow\mathbb{L}$ is a group homomorphism characterized by $\textup{det}\;\mathcal{O}(\vec{x})=\vec{x}$ for each $\vec{x}\in\mathbb{L}$. We have the following observation.

\begin{prop}
\label{x+y=z+u}
Assume that $\vec{x}=\sum\limits_{i=1}^{3}l_{i}\vec{x}_{i}+l\vec{c}$, $\vec{y}=\sum\limits_{i=1}^{3}k_{i}\vec{x}_{i}+k\vec{c}$, $\vec{z}=\sum\limits_{i=1}^{3}\lambda_{i}\vec{x}_{i}+\lambda\vec{c}$, $\vec{u}=\sum\limits_{i=1}^{3}\mu_{i}\vec{x}_{i}+\mu\vec{c}$ are in normal form. Then $[\mathcal{O}(\vec{x})]+[\mathcal{O}(\vec{y})]=[\mathcal{O}(\vec{z})]+[\mathcal{O}(\vec{u})]$ if and only if $l+k=\lambda+\mu$ and $\{l_i,k_i\}=\{\lambda_i,\mu_i\}$ for any $i=1,2,3$.\end{prop}

\begin{proof}
By Proposition \ref{Grothendieck group of line bundles in normal form}, we have
\begin{equation}\label{xy}[\mathcal{O}(\vec{x})]+[\mathcal{O}(\vec{y})]=\sum\limits_{i=1}^{3}([\mathcal{O}(l_{i}\vec{x}_{i})]+[\mathcal{O}(k_{i}\vec{x}_{i})])+(l+k)[\mathcal{O}(\vec{c})] - (l+k+4)[\mathcal{O}]\end{equation}
and
\begin{equation}\label{zu}[\mathcal{O}(\vec{z})]+[\mathcal{O}(\vec{u})]=\sum\limits_{i=1}^{3}([\mathcal{O}(\lambda_{i}\vec{x}_{i})]+[\mathcal{O}(\mu_{i}\vec{x}_{i})])+(\lambda+\mu)[\mathcal{O}(\vec{c})] - (\lambda+\mu+4)[\mathcal{O}].\end{equation}
Hence the sufficiency follows immediately. We only need to prove the necessity in the following.

Note that $\{[\mathcal{O}(\vec{x})]|0\leq \vec{x}\leq \vec{c}\}$ is a basis of $K_0(\mathbb{X})$. By comparing the coefficient of $[\mathcal{O}(\vec{c})]$ in \eqref{xy} and \eqref{zu} we obtain $l+k=\lambda+\mu$. Consequently, one can obtain $$[\mathcal{O}(l_{i}\vec{x}_{i})]+[\mathcal{O}(k_{i}\vec{x}_{i})]-[\mathcal{O}(\lambda_{i}\vec{x}_{i})]-[\mathcal{O}(\mu_{i}\vec{x}_{i})]=0$$
for any $1\leq i\leq3$. It follows that $\{l_i,k_i\}=\{\lambda_i,\mu_i\}$ for $1\leq i\leq3$. We are done.
\end{proof}

\subsection{Isomorphisms between extension bundles}
Kussin, Lenzing and Meltzer \cite{[KLM]} showed that in $\textup{coh}\mbox{-}\mathbb{X}$ each indecomposable vector bundle of rank two is an extension bundle, i.e., has the form $E_L\langle\vec{x}\rangle$; cf.\eqref{extension term}. For the sake of simplicity, we denote $E_{\mathcal{O}}\langle \vec{x}\rangle$ by $E\langle\vec{x}\rangle$ in the rest of this paper. Moreover, all the extension bundles are exceptional in $\textup{coh}\mbox{-}\mathbb{X}$. Hence they are determined by their classes in the Grothendieck group $K_0(\mathbb{X})$.

The following result is a key observation on the feature of extension bundles. We describe which extension bundles are in the same orbit.

\begin{prop}\label{extension bundle which are in the same orbit}
Assume that $\vec{x}, \vec{y}, \vec{z}\in\mathbb{L}$ and $\vec{x}=\sum\limits_{i=1}^{3}l_i\vec{x}_i$ with $0\leq l_i\leq p_i-2$ for $1\leq i\leq 3$. Then $E\langle \vec{x}\rangle\cong E\langle \vec{y}\rangle(\vec{z})$ if and only if one of the following holds:
\begin{itemize}
\item[(i)] $\vec{y}=\vec{x}$ and $\vec{z}=0$;
\item[(ii)] $\vec{y}=l_j\vec{x}_j+\sum\limits_{i\neq j}(p_i-2-l_i)\vec{x}_i$ and $\vec{z}=\sum\limits_{i\neq j}(l_i+1)\vec{x}_i-\vec{c}$ for some $1\leq j\leq 3$.
\end{itemize}
\end{prop}

\begin{proof}
Assume that $\vec{y}=\sum\limits_{i=1}^3 k_i\vec{x}_i$ and $\vec{z}=\sum\limits_{i=1}^3\lambda_i\vec{x}_i+\lambda \vec{c}$ are both in normal forms with $0\leq k_i\leq p_i-2$ for $1\leq i\leq 3$. Then $E\langle \vec{x}\rangle\cong E\langle \vec{y}\rangle (\vec{z})$ if and only if they have the same class in $K_0(\mathbb{X})$ since they are both exceptional in $\textup{coh}\mbox{-}\mathbb{X}$, that is,
\begin{equation}\label{classes expression} [\mathcal{O}(\vec{\omega})]+[\mathcal{O}(\vec{x})]=[\mathcal{O}(\vec{\omega}+\vec{z})]+[\mathcal{O}(\vec{y}+\vec{z})].\end{equation}
Observe that for any $1\leq i\leq 3$, the coefficients of $\vec{x}_i$ in the normal form of $\{\vec{\omega},\vec{x}\}$ are given by
$\{p_i-1, l_i\}$; while  that of
$\{\vec{\omega}+\vec{z},\vec{y}+\vec{z}\}$ are given by $\{\lambda_i-1\;(\textup{mod}\;p_i) , k_i+\lambda_i\;(\textup{mod}\;p_i)\}$.
By Proposition \ref{x+y=z+u}, we obtain that
\eqref{classes expression} holds if and only if \begin{equation}\label{mod equation}\{p_i-1, l_i\}=\{\lambda_i-1\;(\textup{mod}\;p_i) , k_i+\lambda_i\;(\textup{mod}\;p_i)\}\end{equation}
for $1\leq i\leq 3$, and $\vec{\omega}+\vec{x}=(\vec{\omega}+\vec{z})+(\vec{y}+\vec{z})$, i.e., \begin{equation}\label{degree equality} -\vec{x}+\vec{y}+2\vec{z}=0.\end{equation}

The sufficiency is easy to check. In fact, it is trivial for (i); while for (ii) it suffices to observe that $\vec{\omega}+\vec{z}=(p_j-1)\vec{x}_j+\sum_{i\neq j}l_i\vec{x}_i-\vec{c}$ and $\vec{y}+\vec{z}=l_j\vec{x}_j+\sum_{i\neq j}(p_i-1)\vec{x}_i-\vec{c}$.

Now we prove the necessity in the following. By \eqref{degree equality} we obtain
\begin{equation}\label{important equation}
\sum\limits_{i=1}^3(-l_i+k_i+2\lambda_i)\vec{x}_i+2\lambda \vec{c}=0.
\end{equation}
Since $-p_i+2\leq -l_i+k_i+2\lambda_i\leq 3p_i-4$, we have $-l_i+k_i+2\lambda_i=0, p_i$ or $2p_i$.

If there exists some $j$, such that $-l_j+k_j+2\lambda_j=2p_j$, then $\lambda_j>0$ and $p_j+1\leq k_j+\lambda_j\leq 2p_j-3$. Hence neither $\lambda_j-1$ nor $k_j+\lambda_j$
equal to $p_j-1$ modulo $p_j$, contradicting to \eqref{mod equation}.
Hence, $-l_i+k_i+2\lambda_i=0$ or $p_i$ for $i=1,2,3$, which implies
that $\lambda=0$ or $-1$ by (\ref{important equation}).
\begin{itemize}
\item[\emph{Case 1:}] $\lambda=0$, then for any $1\leq i\leq 3$, $-l_i+k_i+2\lambda_i=0$, which implies that $k_i+\lambda_i=l_i-\lambda_i\leq p_i-2$. Then by \eqref{mod equation} we have $\lambda_i=0$ and then $k_i=l_i$ for any $i$. Hence, $\vec{y}=\vec{x}$ and $\vec{z}=0$.
\item[\emph{Case 2:}] $\lambda=-1$, then by (\ref{important equation}), there exists some $j$, such
that $-l_j+k_j+2\lambda_j=0$ and $-l_i+k_i+2\lambda_i=p_i$ for any $i\neq j$. By similar arguments as in Case 1, we have $\lambda_j=0$ and $k_j=l_j$.
For $i\neq j$, by $-l_i+k_i+2\lambda_i=p_i$ we have $\lambda_i>0$. Hence $k_i+\lambda_i=p_i-\lambda_i+l_i\not\equiv l_i\;(\textup{mod}\;p_i)$.
By \eqref{mod equation} we have $l_i=\lambda_i-1$ and then $k_i=p_i-2-l_i$. It follows that $\vec{y}=l_j\vec{x}_j+\sum_{i\neq j}(p_i-2-l_i)\vec{x}_i$
and $\vec{z}=\sum_{i\neq j}(l_i+1)\vec{x}_i-\vec{c}$.
\end{itemize}
This finishes the proof.
\end{proof}

The following corollary reveals when an extension bundle is an Auslander bundle.

\begin{cor}\label{Auslander bundle}
Let $E\langle \vec{x}\rangle$ be an extension bundle with $0\leq \vec{x}\leq \sum\limits_{i=1}^{3}(p_i-2)\vec{x}_i$. Then $E\langle \vec{x}\rangle$ is an Auslander bundle if and only if $\vec{x}=0$ or $\vec{x}=\sum\limits_{i\neq j}(p_i-2)\vec{x}_i$ for some $1\leq j\leq 3$.
\end{cor}

\subsection{$\mathbb{L}$-orbit of extension bundles}
As an application of Proposition \ref{extension bundle which are in the same orbit}, we can give the answer of Question \ref{AA}.

\begin{thm}\label{exact number of orbits}
Let $m$ be the number of even weights in the weight sequence $(p_1, p_2, p_3)$. Then
$$|\mathcal{V}_2/\textup{Pic}(\mathbb{X})|=\left\{
\begin{array}{ll}
\frac{1}{4}\prod\limits_{i=1}^{3}(p_i-1) & m\leq 1\\
\frac{1}{4}(\prod\limits_{i=1}^{3}(p_i-1)+(p_j-1)) & m=2\textup{\;\;with\;\;}p_j\textup{\;odd}\\
\frac{1}{4}(\prod\limits_{i=1}^{3}(p_i-1)+\sum\limits_{i=1}^{3}(p_i-1)) & m=3.
\end{array}\right.$$
\end{thm}

\begin{proof}
Recall from \cite{[KLM]} that any $F\in\mathcal{V}_2$ is an extension bundle, hence $F$ has the form $E_L\langle\vec{x}\rangle$ for some line bundle $L$ and $\vec{x}\in\mathbb{L}$ with $0\leq\vec{x}\leq2\vec{\omega}+\vec{c}=\sum\limits_{i=1}^{3}(p_i-2)\vec{x}_i$. By definition, $F$ lies in the same orbit with $E\langle\vec{x}\rangle$. Denote by $$\mathcal{S}=\{E\langle\vec{x}\rangle | 0\leq \vec{x} \leq \sum\limits_{i=1}^{3}(p_i-2)\vec{x}_i\}.$$
For any $1\leq j\leq 3$, we define a map $$\sigma_j: \mathcal{S}\mapsto\mathcal{S}; \quad
E\langle\sum\limits_{i=1}^{3}l_i\vec{x}_i\rangle \mapsto E\langle l_j\vec{x}_j+\sum\limits_{i\neq j}(p_i-2-l_i)\vec{x}_i\rangle.$$
Then it is easy to see that for any $1\leq i, j\leq 3$, $\sigma_i\sigma_j=\sigma_j\sigma_i$ and $\sigma_j^2=\textup{id}$, where $\textup{id}$ is the identity map on $\mathcal{S}$. Hence $G=\{\textup{id}, \sigma_1, \sigma_2, \sigma_3\}$ is a Klein four-group which acts on the set $\mathcal{S}$. Moreover, by Proposition \ref{extension bundle which are in the same orbit}, any two extension bundles from $\mathcal{S}$ are in the same $\mathbb{L}$-orbit if and only if they are in the same $G$-orbit. Denote $n=|\mathcal{V}_2/\mathbb{L}|$, then by Burnside Formula,
$$n=\frac{1}{|G|}\sum\limits_{\sigma\in G}|\mathcal{S}^{\sigma}|,$$ where $|X|$ denotes the order of the set $X$ and $\mathcal{S}^{\sigma}=\{s\in \mathcal{S}|\sigma(s)=s\}$ is the subset of fixed points. Obviously, $|G|=4$ and $|\mathcal{S}^{\textup{id}}|=|\mathcal{S}|=\prod\limits_{i=1}^{3}(p_i-1)$. Moreover, for any $1\leq j\leq 3$, by Proposition \ref{extension bundle which are in the same orbit},
$$\sigma_j(E\langle\sum\limits_{i=1}^{3}l_i\vec{x}_i\rangle)=E\langle\sum\limits_{i=1}^{3}l_i\vec{x}_i\rangle$$
if and only $p_i-2-l_i=l_i$ for any $i\neq j$, which implies that $p_i$ is even and $l_i=\frac{p_i-2}{2}$ for $i\neq j$, hence in this case, $|\mathcal{S}^{\sigma_j}|=p_j-1$. Now we consider the following three cases with respect to the number $m$ of even weights in the weight sequence $(p_1, p_2, p_3)$.
\begin{itemize}
\item[(i)]  If $m\leq 1$, then for any $i$, $|\mathcal{S}^{\sigma_i}|=0$. Hence $n=\frac{1}{4}\prod\limits_{i=1}^{3}(p_i-1)$;
\item[(ii)] If $m=2$ and $p_j$ is the unique odd weight, then $|\mathcal{S}^{\sigma_j}|=p_j-1$ and $|S^{\sigma_i}|=0$ for $i\neq j$.
Hence $n=\frac{1}{4}(\prod\limits_{i=1}^{3}(p_i-1)+(p_j-1))$;
\item[(iii)] If $m=3$, then for any $i$, $|\mathcal{S}^{\sigma_i}|=p_i-1$.
Hence $n=\frac{1}{4}(\prod\limits_{i=1}^{3}(p_i-1)+\sum\limits_{i=1}^{3}(p_i-1))$.
\end{itemize}
\end{proof}

Recall that the group action of $\textup{Pic}(\mathbb{X})$ on $\mathcal{V}_2$ is transitive if there exists a unique orbit, i.e., $|\mathcal{V}_2/\textup{Pic}(\mathbb{X})|=1$. As an immediate consequence, we have the following result.

\begin{cor}
The Picard group of $\mathbb{X}$ acts transitively on the set of extension bundles if and only if $\mathbb{X}$ has weight type $(2,2,2), (2,2,3)$ or $(2,3,3)$.
\end{cor}

\subsection{$\tau$-orbit of extension bundles}
In this subsection we will calculate the $\tau$-orbits of extension bundles. We always assume $\mathbb{X}$ is not of tubular type in this subsection, since there are infinity many $\tau$-orbits of line bundles and extension bundles for tubular case.

Let $\mathbb{L}'=\{\vec{x}\,|\,0\leq\vec{x}\leq\sum\limits_{i=1}^{3}(p_i-2)\vec{x}_i\}$ be a subset of $\mathbb{L}$. Then there is a canonical surjection \begin{equation}\label{pi}
    \pi:\;\mathscr{L}\times \mathbb{L}'\to\mathcal{V}_2,\;(L,\vec{x})\mapsto E_L\langle\vec{x}\rangle,
\end{equation}
which induces a surjection $$\pi':\;\mathscr{L}/{\langle\tau\rangle} \times \mathbb{L}'\to
\mathcal{V}_2/{\langle\tau\rangle}.$$
We will calculate the number of $\tau$-orbits of extension bundles via the cardinality of $\mathscr{L}/{\langle\tau\rangle}\times \mathbb{L}'$.

Note that $|\mathbb{L}'|=\prod\limits_{i=1}^{3}(p_i-1)$.
Let $\delta:\;\mathbb{L}\rightarrow\mathbb{Z}$ be the group homomorphism defined by $\delta(\vec{x}_i)=\frac{p}{p_i}$, where $p:=\text{l.c.m.}(p_1, p_2, p_3)$ denotes the least common multiple of $p_1, p_2, p_3$. Recall from \cite[Lemma 4.19]{Lenzing2011Weighted} that
the number of $\tau$-orbits of line bundles are given by
\begin{equation}\label{gongshi2.7}
|\mathscr{L}/{\langle\tau\rangle}|=[\mathbb{L}:\mathbb{Z}\vec{\omega}]=\frac{1}{p}\delta(\vec{\omega})\prod\limits_{i=1}^{3}p_i=(1-\sum\limits_{i=1}^{3}\frac{1}{p_i})\prod\limits_{i=1}^{3}p_i.\end{equation}

In order to determine the fiber of $\pi'$,  we define maps $\sigma_j$ ($1\leq j\leq 3$) as follows:
\begin{align*}
\sigma_j: \;&\mathscr{L}\times\mathbb{L}'\to\mathscr{L}\times\mathbb{L}';\\
&(L,\sum\limits_{i=1}^{3}l_i\vec{x}_i) \mapsto (L(\sum\limits_{i\neq j}l_i\vec{x}_i-\vec{x}_j),l_j\vec{x}_j+\sum\limits_{i\neq j}(p_i-2-l_i)\vec{x}_i).\end{align*}

\begin{lem}\label{CommutativeDiagram}
For any  $1\leq j\leq 3$, $\pi\circ \sigma_j=\tau^{-1}\circ \pi$, i.e., we have the following commutative diagram.
$$\xymatrix{\mathscr{L}\times \mathbb{L}'\ar[r]^{\sigma_j}\ar[d]^{\pi}&\mathscr{L}\times \mathbb{L}'\ar[d]^{\pi}\\
\mathcal{V}_2\ar[r]^{\tau^{-1}}&\mathcal{V}_2.}$$
\end{lem}

\begin{proof}
For any $(L,\sum\limits_{i=1}^{3}l_i\vec{x}_i) \in \mathscr{L}\times\mathbb{L}'$, we have
\begin{align*}
\pi[\sigma_j(L,\sum\limits_{i=1}^{3}l_i\vec{x}_i)]&=\pi(L(\sum\limits_{i\neq j}l_i\vec{x}_i-\vec{x}_j),l_j\vec{x}_j+\sum\limits_{i\neq j}(p_i-2-l_i)\vec{x}_i)\nonumber\\
&=E_{L}\langle l_j\vec{x}_j+\sum\limits_{i\neq j}(p_i-2-l_i)\vec{x}_i\rangle(\sum\limits_{i\neq j}l_i\vec{x}_i-\vec{x}_j)\nonumber\\
&=E_L\langle\sum\limits_{i=1}^{3}l_i\vec{x}_i\rangle(-\vec{\omega})\\&=\tau^{-1}[\pi(L,\sum\limits_{i=1}^{3}l_i\vec{x}_i)],
\end{align*}
where the third equation follows from  Proposition \ref{extension bundle which are in the same orbit}. Hence $\pi\circ \sigma_j=\tau^{-1}\circ \pi$.
\end{proof}

Obviously, for $1\leq i\leq 3$, $\sigma_j$ induces an automorphism on $\mathscr{L}/\langle\tau\rangle\times \mathbb{L}'$, still denoted by $\sigma_j$.

\begin{lem}
Let $G$ be the automorphism subgroup of $\mathscr{L}/\langle\tau\rangle\times \mathbb{L}'$ generated by $\sigma_1,\sigma_2,\sigma_3$. Then $G$ is a Klein four-group, i.e.,
\begin{itemize}
    \item[(1)] $\sigma_i^2=\textup{id}$ for any $1\leq i\leq 3$;
     \item[(2)] $\sigma_{i}\sigma_j=\sigma_k$ for any $\{i,j,k\}=\{1,2,3\}$.
\end{itemize}
\end{lem}

\begin{proof}
For any $L\in \mathscr{L}$ and $\vec{x}=\sum\limits_{t=1}^{3}l_t\vec{x}_t\in\mathbb{L}'$, we have
$$\sigma_i(L,\vec{x})=(L(\sum\limits_{t\neq i}l_t\vec{x}_t-\vec{x}_i),l_i\vec{x}_i+\sum\limits_{t\neq i}(p_t-2-l_t)\vec{x}_t).$$
Hence,
\begin{align*}
\sigma_i[\sigma_i(L,\vec{x})]&=(L(\sum\limits_{t\neq i}l_t\vec{x}_t-\vec{x}_i)(\sum\limits_{t\neq i}(p_t-2-l_t)\vec{x}_t-\vec{x}_i),\sum\limits_{t=1}^{3}l_t\vec{x}_t)\\
&=(L(2\vec{\omega}),\vec{x});\\
\sigma_{j}[\sigma_i(L,\vec{x})]&=\sigma_{j}(L(\sum\limits_{t\neq i}l_t\vec{x}_t-\vec{x}_i),l_i\vec{x}_i+\sum\limits_{t\neq i}(p_t-2-l_t)\vec{x}_t) \\
&=(L(l_{j}\vec{x}_{j}+l_{k}\vec{x}_{k}-\vec{x}_{i})[(p_k-2-l_{k})\vec{x}_{k}+l_{i}\vec{x}_{i}-\vec{x}_{j}],l_k\vec{x}_k+\sum\limits_{t\neq k}(p_t-2-l_t)\vec{x}_t)\\
&=(L[(p_k-2)\vec{x}_{k}+\sum\limits_{t\neq k}(l_t-1)\vec{x}_{t}],l_k\vec{x}_k+\sum\limits_{t\neq k}(p_t-2-l_t)\vec{x}_t)\\
&=(L(\vec{\omega})(\sum\limits_{t\neq k}l_t\vec{x}_t-\vec{x}_k),l_k\vec{x}_k+\sum\limits_{t\neq k}(p_t-2-l_t)\vec{x}_t)\\
&=\sigma_k(L(\vec{\omega}),\vec{x}).
\end{align*}
Therefore, in $\mathscr{L}/\langle\tau\rangle\times \mathbb{L}'$ we have
$$\sigma_i[\sigma_i(\overline{L},\vec{x})]=(\overline{L(2\vec{\omega})},\vec{x})=(\overline{L},\vec{x}) \text{\quad and\quad}
\sigma_{j}[\sigma_i(\overline{L},\vec{x})]=\sigma_k(\overline{L(\vec{\omega})},\vec{x})=\sigma_k(\overline{L},\vec{x}).$$
Then we are done.
\end{proof}

\begin{lem}
The group $G$ acts freely on $\mathscr{L}/\langle\tau\rangle\times \mathbb{L}'$.
\end{lem}

\begin{proof}
We need to show that all the elements $(\overline{L},\vec{x})\in \mathscr{L}/\langle\tau\rangle\times \mathbb{L}'$ are not fixed by $\sigma_j$ for any $1\leq j\leq 3$.

For contradiction we assume $\sigma_j(\overline{L},\vec{x})=(\overline{L},\vec{x})$ for some $j$. Assume $\vec{x}=\sum\limits_{i=1}^{3}l_i\vec{x}_i$. Then we have
\begin{align}\label{29-1}
    \vec{x}=l_j\vec{x}_j+\sum\limits_{i\neq j}(p_i-2-l_i)\vec{x}_i
\end{align} and
\begin{align}\label{29-2}
    \sum\limits_{i\neq j}l_i\vec{x}_i-\vec{x}_j\in\mathbb{Z}\vec{\omega}.
\end{align}
By (\ref{29-1}), we have $l_i=p_i-2-l_i$, that is, $l_i+1=\frac{p_i}{2}$ for $i\neq j$. By (\ref{29-2}), $\sum_{i\neq j}l_i\vec{x}_i-\vec{x}_j=r\vec{\omega}$ for some $r\in\mathbb{Z}$, that is,
$\sum_{i\neq j}(l_i+r)\vec{x}_i+(r-1)\vec{x}_j=r\vec{c}$.
Hence $p_i|(l_i+r)$ for $i\neq j$, $p_j|(r-1)$ and $\sum_{i\neq j}\frac{l_i+r}{p_i}+\frac{r-1}{p_j}=r$. It implies that $(r-1)(1-\sum\limits_{i=1}^{3}\frac{1}{p_i})=0$. Therefore, $r=1$ or $\sum\limits_{i=1}^{3}\frac{1}{p_i}=1$, which contradict to $p_i|(l_i+r)$ and $\delta(\vec{\omega})\neq 0$, respectively.
We are done.
\end{proof}

\begin{prop}\label{210}
Assume that $\delta(\vec{\omega})\neq 0$. Then the number of $\tau$-orbits of extension bundles is given by
$\frac{1}{4}(1-\sum\limits_{i=1}^{3}\frac{1}{p_i})\prod\limits_{i=1}^{3}p_i(p_i-1)$.
\end{prop}

\begin{proof}
For any extension bundle $E_{L}\langle\vec{x}\rangle$, assume $\pi(L',\vec{x}')=E_{L}\langle\vec{x}\rangle$, then by Proposition \ref{extension bundle which are in the same orbit} and Lemma \ref{CommutativeDiagram}, we have $(L',\vec{x}')=(L,\vec{x})$ or $\sigma_j(L,\vec{x})$ for some
$1\leq j\leq 3$. Consequently, $\pi'$ induces a bijection between the $G$-orbits of  $\mathscr{L}/{\langle\tau\rangle} \times \mathbb{L}'$ with $\mathcal{V}_2/{\langle\tau\rangle}.$

Then by \eqref{gongshi2.7} we have
\begin{align*}
    |\mathcal{V}_2/{\langle\tau\rangle}|
    =&[\mathscr{L}/{\langle\tau\rangle}\times \mathbb{L}']|/|G|\\
    =&\frac{1}{4}[\mathbb{L}:\mathbb{Z}\vec{\omega}]\prod\limits_{i=1}^{3}(p_i-1)\\
    =&\frac{1}{4}(1-\sum\limits_{i=1}^{3}\frac{1}{p_i})\prod\limits_{i=1}^{3}p_i(p_i-1).
\end{align*}
\end{proof}

\section{Invariants and stabilities for extension bundles}\label{SectionApp}
In this section, we will give some invariants for extension bundles and investigate their stabilities.

\subsection{Projective cover and injective hull}

Denote by $\mathfrak{P}(E)$ (resp. $\mathfrak{I}(E)$) the projective
cover (resp. injective hull) of a vector bundle $E$. The following result was conjectured by Professor Helmut Lenzing during private communication with the second-named author, indicating that the projective cover and injective hull are both invariants for extension bundles.

\begin{prop}\label{projective cover determines extension bundle}
Assume that $E$ and $F$ are extension bundles in $\textup{vect}\mbox{-}\mathbb{X}$, then the following are equivalent.

(1) $E\cong F$;

(2) $\mathfrak{P}(E)\cong \mathfrak{P}(F)$;

(3) $\mathfrak{I}(E)\cong \mathfrak{I}(F)$.
\end{prop}

\begin{proof}
We only show that the statements (1) and (2) are equivalent, the equivalence of (1) and (3) is quite similar.

If (1) holds, then obviously (2) is true. On the other hand, we assume that $E$ and $F$ share the same projective covers. We will show that $E\cong F$. Without loss of generality by the action of line bundle twist, we assume that $E=E\langle\vec{x}\rangle$ and $F=E\langle \vec{y}\rangle(\vec{z})$, where $\vec{x}=\sum\limits_{i=1}^{3}a_i\vec{x}_i$, $\vec{y}=\sum\limits_{i=1}^{3}b_i\vec{x}_i$ and $\vec{z}=\sum\limits_{i=1}^{3}l_i\vec{x}_i+l\vec{c}$ are in normal forms, and $0\leq a_i, b_i\leq p_i-2$ for any $1\leq i\leq 3$. By \cite[Theorem 4.6]{[KLM]}, the projective covers of $E$ and $F$ are given by
$$\mathfrak{P}(E)=\mathcal{O}(\vec{\omega})\oplus(\bigoplus\limits_{i=1}^{3}\mathcal{O}(\vec{x}-(a_i+1)\vec{x}_i))$$
and
$$\mathfrak{P}(F)=\mathcal{O}(\vec{\omega}+\vec{z})\oplus(\bigoplus\limits_{i=1}^{3}\mathcal{O}(\vec{y}+\vec{z}-(b_i+1)\vec{x}_i)).$$
Then the assumption $\mathfrak{P}(E)\cong \mathfrak{P}(F)$ implies that $\mathcal{O}(\vec{\omega}+\vec{z})$ is a direct summand of $\mathfrak{P}(E)$. There are two cases to consider:
\begin{itemize}
\item[\emph{Case 1:}] $\mathcal{O}(\vec{\omega}+\vec{z})=\mathcal{O}(\vec{\omega})$, it follows that $\vec{z}=0$. Then by comparing the other three direct summands of $\mathfrak{P}(E)$ and $\mathfrak{P}(F)$, and noticing that $0\leq a_i, b_i\leq p_i-2$ for any $1\leq i\leq 3$, it is easy to see that $a_i=b_i$ for any $1\leq i\leq 3$, that is, $\vec{x}=\vec{y}$. Hence $E\cong F$.
\item[\emph{Case 2:}] $\mathcal{O}(\vec{\omega}+\vec{z})=\mathcal{O}(\vec{x}-(a_j+1)\vec{x}_j)$ for some $1\leq j\leq 3$, then $\vec{\omega}+\vec{z}=\vec{x}-(a_j+1)\vec{x}_j$, that is, $\sum\limits_{i=1}^{3}(l_i-1)\vec{x}_i+(l+1)\vec{c}=\sum\limits_{i\neq j}a_i\vec{x}_i-\vec{x}_j$. Notice that for $1\leq i\leq 3$, $0\leq l_i\leq p_i-1$ and $0\leq a_i, b_i\leq p_i-2$. We get $l_j=0$, $l_i=a_i+1$ for $i\neq j$, and $l=-1$. Hence
\begin{equation}\label{expression of z}
\vec{z}=\sum\limits_{i\neq j}(a_i+1)\vec{x}_i-\vec{c}.
\end{equation}
By considering the direct summand $\mathcal{O}(\vec{\omega})$ of $\mathfrak{P}(E)$, we know that there exists some $1\leq k\leq 3$, such that
\begin{equation}\label{equation of the summand w}\mathcal{O}(\vec{\omega})=\mathcal{O}(\vec{y}+\vec{z}-(b_k+1)\vec{x}_k).\end{equation}
We claim that $k=j$. Otherwise, by comparing the coefficients of $\vec{x}_j$ for the determinants of both sides of (\ref{equation of the summand w}), we get $b_j\equiv p_j-1(\textup{mod}\;p_j)$, which is a contradiction to the assumption $0\leq b_j\leq p_j-2$, as claimed. Thus we have $\vec{\omega}=\vec{y}+\vec{z}-(b_j+1)\vec{x}_j$. By comparing the normal forms of both sides, we then obtain $b_i=p_i-2-a_i$ for any $i\neq j$.

Similarly, by considering the direct summand $\mathcal{O}(\vec{x}-(a_i+1)\vec{x}_i)$ of $\mathfrak{P}(E)$ for any $i\neq j$, we get $\mathcal{O}(\vec{x}-(a_i+1)\vec{x}_i)=\mathcal{O}(\vec{y}+\vec{z}-(b_k+1)\vec{x}_k)$ for some $k\neq j$. It follows that $\vec{x}-(a_i+1)\vec{x}_i=\vec{y}+\vec{z}-(b_k+1)\vec{x}_k$, where $i,k\neq j$. Then by comparing the coefficients of $\vec{x}_j$, we get $b_j=a_j$. Therefore, we finally obtain that
\begin{equation}\label{expression of new y}
\vec{y}=a_j\vec{x}_j+\sum\limits_{i\neq j}(p_i-2-a_i)\vec{x}_i.
\end{equation}
Combining with (\ref{expression of z}) and (\ref{expression of new y}), and according to Proposition \ref{extension bundle which are in the same orbit}, we have $E\langle \vec{x}\rangle\cong E\langle
\vec{y}\rangle(\vec{z})$, that is, $E\cong F$. We are done.
\end{itemize}
\end{proof}

\subsection{Semistable vector bundles}
In this subsection we describe all the stable and semistable vector bundles of rank two. Recall that the degree function $\textup{deg}$ is the linear form on $K_0(\mathbb{X})$ which is characterized by $\textup{deg}(\mathcal{O}(\vec{x}))=\delta(\vec{x})$. By \cite{Lenzing2011Weighted} each non-zero object $E$ has a well-defined slope in $\overline{\mathbb{Q}}=\mathbb{Q}\cup\{\infty\}$ given by $\mu(E):=\frac{\textup{deg}(E)}{\textup{rk}(E)}$. A non-zero object $E$ is called \emph{semistable} (resp. \emph{stable}) if $\mu(F)\leq \mu(E)$ (resp. $\mu(F)<\mu(E)$) for any non-trivial subobject $F$ of $E$.

\begin{prop}\label{semistable}
Let $L$ be a line bundle, and let $\vec{x}=\sum\limits_{i=1}^3l_i\vec{x}_i$ be an element from $\mathbb{L}$ with $0\leq \vec{x} \leq \sum\limits_{i=1}^{3}(p_i-2)\vec{x}_i$. Then $E_L\langle\vec{x}\rangle$ is semistable if and only if $$\delta(\vec{\omega})\leq \delta(\vec{x})\leq \delta(\vec{\omega}+2(l_i+1)\vec{x}_i),\; \text{\;for\;any\;} 1\leq i\leq 3.$$ Moreover, it is stable if and only if both strict inequalities hold.
\end{prop}

\begin{proof}
Let $F$ be a non-trivial subobject of $E_L\langle\vec{x}\rangle$ and $\iota: F\hookrightarrow E_L\langle\vec{x}\rangle$ be the natural embedding. We know that $1\leq \textup{rk}(F)\leq 2$.

\emph{Case 1:} $\textup{rk}(F)=2$; then $\textup{Coker}(\iota)$ has rank zero, hence it is a torsion sheaf and $\textup{deg}(\textup{Coker}(\iota))>0$. It follows that $\textup{deg}(F)<\textup{deg}(E_L\langle \vec{x}\rangle)$, and hence $\mu(F)<\mu(E_L\langle \vec{x}\rangle)$.

\emph{Case 2:} $\textup{rk}(F)=1$; then $F$ is a line bundle, hence $\iota$ factors through the projective cover $\mathfrak{P}(E_L\langle \vec{x}\rangle)$ of $E_L\langle \vec{x}\rangle$. Hence the slope of $F$ is smaller than or equal to the slope of a direct summand of $\mathfrak{P}(E_L\langle\vec{x}\rangle)$.

Therefore, $E_L\langle\vec{x}\rangle$ is semistable if and only if $\mu(L_i)\leq \mu(E_L\langle \vec{x}\rangle)$ for each indecomposable direct summand $L_i$ of $\mathfrak{P}(E_L\langle\vec{x}\rangle)$. Recall that
$$\mathfrak{P}(E_L\langle \vec{x}\rangle)=L(\vec{\omega})\oplus(\bigoplus\limits_{i=1}^{3}L(\vec{x}-(l_i+1)\vec{x}_i))$$ and there are exact sequences
$$0\to L(\vec{\omega})\to E_L\langle \vec{x}\rangle\to L(\vec{x})\to 0$$
and
$$0\to L(\vec{x}-(l_i+1)\vec{x}_i))\to E_L\langle \vec{x}\rangle\to L(\vec{\omega}+(l_i+1)\vec{x}_i))\to 0.$$
Hence by seesaw property we get $E_L\langle\vec{x}\rangle$ is semistable if and only if $\mu(L(\vec{\omega}))\leq \mu(L(\vec{x}))$ and $\mu(L(\vec{x}-(l_i+1)\vec{x}_i))\leq \mu(L(\vec{\omega}+(l_i+1)\vec{x}_i))$, or equivalently $\delta(\vec{\omega})\leq \delta(\vec{x}) \leq \delta(\vec{\omega}+2(l_i+1)\vec{x}_i)$, for any $1\leq i\leq 3$.

Similarly, $E_L\langle \vec{x}\rangle$ is stable if and only if $\delta(\vec{\omega})< \delta(\vec{x}) < \delta(\vec{\omega}+2(l_i+1)\vec{x}_i)$ for any $1\leq i\leq 3$. This finishes the proof.
\end{proof}

As an immediate consequence, we can obtain the following well-known result, c.f. \cite{Lenzing2011Weighted}.
\begin{cor} Let $\mathbb{X}$ be a weighted projective line. The following hold:
\begin{itemize}
  \item [(i)] if $\delta(\vec{\omega})<0$, then each extension bundle is stable;
  \item [(ii)] if $\delta(\vec{\omega})=0$, then each extension bundle is semistable;
  \item [(iii)] if $\delta(\vec{\omega})>0$, then there exists non-semistable extension bundle.
\end{itemize}
\end{cor}

\begin{proof}
Let $E_L\langle\vec{x}\rangle$ be an extension bundle with $\vec{x}=\sum\limits_{i=1}^3l_i\vec{x}_i$, where $0\leq i\leq p_i-2$.
Note that for any $1\leq i\leq 3$,
\begin{align*}(\vec{\omega}+2(l_i+1)\vec{x}_i)-\vec{x}&=\vec{c}-\sum\limits_{1\leq i\leq 3}\vec{x}_i+(l_i+2)\vec{x}_i-\sum\limits_{j\neq i}l_j\vec{x}_j\\
&=\vec{c}+(l_i+1)\vec{x}_i-\sum\limits_{j\neq i}(l_j+1)\vec{x}_j\\
&\geq\vec{c}+(l_i+1)\vec{x}_i-\sum\limits_{j\neq i}(p_j-1)\vec{x}_j\\
&=-\vec{\omega}+l_i\vec{x}_i.
\end{align*}
Then the results follow from Proposition \ref{semistable} directly.
\end{proof}

\begin{exm}
Let $\mathbb{X}$ be a weighted projective line of tubular type $(p_1, p_2, p_3)$. Let $E_L\langle\vec{x}\rangle$ be an extension bundle with $0\leq \vec{x}=\sum\limits_{i=1}^3l_i\vec{x}_i \leq \sum\limits_{i=1}^{3}(p_i-2)\vec{x}_i$. The following statements hold:
\begin{itemize}
  \item [(i)] if $\mathbb{X}$ has weight type $(3,3,3)$, then $0\leq\vec{x}\leq \vec{x}_1+\vec{x}_2+\vec{x}_3$, and $E_L\langle\vec{x}\rangle$ is stable if and only if  $\sum\limits_{i=1}^3l_i=1$ or 3;
   \item [(ii)] if $\mathbb{X}$ has weight type $(2,3,6)$, then $0\leq\vec{x}\leq \vec{x}_2+4\vec{x}_3$, and $E_L\langle\vec{x}\rangle$ is stable if and only if  $l_3\neq 0$ or 4;
   \item [(ii)] if $\mathbb{X}$ has weight type $(2,4,4)$, then $0\leq \vec{x}\leq 2\vec{x}_2+2\vec{x}_3$, and $E_L\langle\vec{x} \rangle$ is stable if and only if $l_i=1$ for $i=2$ or 3.
\end{itemize}
Combining with Corollary \ref{Auslander bundle}, we know that an extension bundle is stable if and only if it is not an Auslander bundle.
\end{exm}

\section{Tilting objects in $\underline{\textup{vect}}\mbox{-}\mathbb{X}$ consisting of extension bundles}\label{Section4}

Kussin, Lenzing, and Meltzer in \cite{[KLM]} showed that the category $\textup{vect}\mbox{-}\mathbb{X}$ is a Forbenius category with the system $\mathscr{L}$ of all line bundles as the indecomposable projective-injective objects, then the induced stable category $\underline{\textup{vect}}\mbox{-}\mathbb{X}$ is a triangulated category by a general result of Happel \cite{[H]}.

In this section, we investigate the tilting objects in the triangulated category $\underline{\textup{vect}}\mbox{-}\mathbb{X}$. For weight type $(2, p, q)$, we construct a tilting object in $\underline{\textup{vect}}\mbox{-}\mathbb{X}$ consisting only of Auslander bundles, which gives a positive answer to Question \ref{bB}. As its application, we generalize the result of Theorem 4.5 in \cite{Lenzing2011Weighted}.

\subsection{A useful distinguished triangle}

In this subsection we generalize a distinguished triangle which will be useful in the proof of the construction of tilting objects. First we collect some basic properties of the triangulated category $\underline{\textup{vect}}\mbox{-}\mathbb{X}$ for weight type $(p_1,p_2,p_3)$ from \cite{[KLM]}. Keeping the notation in \cite{[KLM]}, we denote by $\underline{\textup{Hom}}(X,Y)$ the homomorphism space between $X$ and $Y$ in $\underline{\textup{vect}}\mbox{-}\mathbb{X}$, and set $$\bar{x}_j=\vec{\omega}+\vec{x}_j=\vec{c}-\sum\limits_{i\neq j}\vec{x}_i \text{\quad for\quad} 1\leq j\leq3.$$

\begin{lem}\label{BasicPropInKLM}\textup{\cite[Corollaries 4.14 and B.2]{[KLM]}}
Let $\mathbb{X}$ be a weighted projective line of weight type $(p_1,p_2,p_3)$. Let $E$ be an Auslander bundle and $\vec{x}\in\mathbb{L}$. Then
\begin{itemize}
\item[(1)] $\underline{\textup{Hom}}(E,E(\vec{x}))\neq 0$ if and only if $\vec{x}\in \{0, \bar{x}_1, \bar{x}_2, \bar{x}_3\}$, and in this case $\underline{\textup{Hom}}(E, E(\vec{x}))$ is isomorphic to $\bf{k}$.
\item[(2)] The two-fold suspension $[2]$ of $\underline{\textup{vect}}\mbox{-}\mathbb{X}$ is equivalent to the line bundle twist $X\mapsto X(\vec{c})$ by the canonical element.
\end{itemize}
\end{lem}

Using Proposition \ref{extension bundle which are in the same orbit}, we give the following expression for extension bundles under the suspension functor [1].
\begin{prop}\label{special shift}
Let $\mathbb{X}$ be a weighted projective line of weight type $(p_1,p_2,p_3)$. Assume that $\vec{x}=\sum\limits_{i=1}^3l_i\vec{x}_i$ with $0\leq l_i \leq p_i-2$ for $1\leq i\leq 3$. Then for any $i$,
\begin{equation}\label{Shift[1]}
E\langle\vec{x}\rangle[1]=E\langle (p_i-2-l_i)\vec{x}_i+\sum\limits_{j\neq i}l_j\vec{x}_j\rangle((l_i+1)\vec{x}_i).\end{equation}
In particular, for any $1\leq i\leq 3$, we have
\begin{equation}\label{extension bundle shift to be auslander bundle}
E\langle(p_i-2)\vec{x}_i\rangle[1]=E((p_i-1)\vec{x}_i).
\end{equation}
\end{prop}

\begin{proof}
By \cite[Corollary 4.8]{[KLM]}, we get
$$E\langle\vec{x}\rangle[1]=E\langle2\vec{\omega}+\vec{c}-\vec{x}\rangle(\vec{x}-\vec{\omega})=E\langle\sum\limits_{i=1}^{3}(p_i-2-l_i)\vec{x}_i\rangle(\sum\limits_{i=1}^{3}(l_i+1)\vec{x}_i-\vec{c}),$$ which equals to $E\langle (p_i-2-l_i)\vec{x}_i+\sum\limits_{j\neq i}l_j\vec{x}_j\rangle((l_i+1)\vec{x}_i)$ for any $1\leq i\leq 3$ by Proposition \ref{extension bundle which are in the same orbit}. This finishes the proof.
\end{proof}

Note that when $p_1=2$, (\ref{Shift[1]}) implies $E\langle\vec{x}\rangle[1]=E\langle\vec{x}\rangle(\vec{x}_1)$.

\begin{lem}\label{important triangles lemma}\textup{\cite[Proposition 4.20]{[KLM]}}
Assume that $\vec{x}=\sum\limits_{i=1}^3l_i\vec{x}_i$ is in normal form and $0\leq\vec{x}<\vec{x}+\vec{x}_i\leq \sum\limits_{i=1}^3(p_i-2)\vec{x}_i$. Then there is a distinguished triangle in $\underline{\textup{vect}}\mbox{-}\mathbb{X}$
\begin{equation}\label{important triangles}
E\langle \vec{x}\rangle \to E\langle \vec{x}+\vec{x}_i\rangle \to E\langle\vec{x}-l_i\vec{x}_i\rangle((l_i+1)\vec{x}_i)\to E\langle \vec{x}\rangle[1].
\end{equation}
\end{lem}

More general, we have the following result:
\begin{lem}\label{important triangles lemma n2}
Assume that $\vec{x}=\sum\limits_{i=1}^3l_i\vec{x}_i$ is in normal form and $0\leq\vec{x}<\vec{x}+b\vec{x}_i\leq \sum\limits_{i=1}^3(p_i-2)\vec{x}_i$. Then there is a distinguished triangle in $\underline{\textup{vect}}\mbox{-}\mathbb{X}$
\begin{equation}\label{important triangles n2}
E\langle\vec{x}\rangle \to E\langle\vec{x}+b\vec{x}_i\rangle\to E\langle\vec{x}-l_i\vec{x}_i+(b-1)\vec{x}_i\rangle((l_i+1)\vec{x}_i)\to E\langle\vec{x}\rangle[1].
\end{equation}
\end{lem}

\begin{proof}
We prove \eqref{important triangles n2} by induction on $b$. During the proof  a triangle $X\to Y\to Z\to X[1]$ will be  denoted by $X\to Y\to Z$ for convenience.

If $b=1$, then (\ref{important triangles n2}) coincides with (\ref{important triangles}), there is nothing to prove. Now assume (\ref{important triangles n2}) holds for $b$ and plan to show it also holds for $b+1$. In fact, by octahedral axiom we have the following commutative diagram with exact triangles:
$$\xymatrix{
  E\langle \vec{x} \rangle \ar@{=}[d] \ar[r]  & E\langle \vec{x}+b\vec{x}_i\rangle \ar[d] \ar[r]  &  F_1\ar[d]\\
 E\langle \vec{x}\rangle \ar[r]  & E\langle \vec{x}+(b+1)\vec{x}_i\rangle \ar[d]  \ar[r] & F\ar[d] \\
 &F_2  \ar@{=}[r]  &  F_2.  }$$
By induction we know that
$$F_1=E\langle\vec{x}-l_i\vec{x}_i+(b-1)\vec{x}_i\rangle((l_i+1)\vec{x}_i) \textup{\quad and\quad} F_2=E\langle\vec{x}-l_i\vec{x}_i\rangle((b+l_i+1)\vec{x}_i).$$
By Proposition \ref{special shift} we have $$F_1[1]=E\langle\vec{x}-l_i\vec{x}_i+(p_i-2-(b-1))\vec{x}_i\rangle((b+l_i+1)\vec{x}_i).$$
Hence $$\underline{\textup{Ext}}^1(F_2, F_1)=\underline{\textup{Hom}}(F_2, F_1[1])=\underline{\textup{Hom}}(E\langle \vec{x}-l_i\vec{x}_i\rangle, E\langle\vec{x}-l_i\vec{x}_i+(p_i-2-(b-1))\vec{x}_i\rangle),$$
which has dimension one by \cite[Lemma 4.5]{[KLM]}. Moreover, by Lemma \ref{important triangles lemma} there is a non-split triangle:
\begin{equation}\label{third column}
E\langle \vec{x}-l_i\vec{x}_i+(b-1)\vec{x}_i\rangle \to E\langle \vec{x}-l_i\vec{x}_i+b\vec{x}_i\rangle \to E\langle \vec{x}-l_i\vec{x}_i\rangle(b\vec{x}_i).
\end{equation}
Therefore, the triangle $F_1\to F\to F_2$ has the form (\ref{third column}) up to a degree shift by $(l_i+1)\vec{x}_i$.
Hence $F=E\langle \vec{x}-l_i\vec{x}_i+b\vec{x}_i\rangle((l_i+1)\vec{x}_i)$, then we are done.
\end{proof}

The following two triangles (for $1\leq b\leq p_i-2$) are special cases of (\ref{important triangles}) and (\ref{important triangles n2}) respectively, which will be used frequently later:
$$E\langle (b-1)\vec{x}_i\rangle  \to E\langle b\vec{x}_i\rangle \to E(b\vec{x}_i)\to E[1]$$ and
$$ E \to E\langle b\vec{x}_i\rangle \to E\langle
(b-1)\vec{x}_i\rangle(\vec{x}_i)\to E[1].$$

\subsection{Tilting objects in $\underline{\textup{vect}}\mbox{-}\mathbb{X}$}

Recall that a basic object $T$ in $\underline{\textup{vect}}\mbox{-}\mathbb{X}$ is \emph{tilting} if the following two conditions hold:
\begin{itemize}
\item[(i)] $T$ is extension-free, i.e. for any $n\neq 0$, $\underline{\textup{Hom}}(T,T[n])=0$;
\item[(ii)] $T$ generates $\underline{\textup{vect}}\mbox{-}\mathbb{X}$, denoted by $\langle X\rangle=\underline{\textup{vect}}\mbox{-}\mathbb{X}$, i.e. the smallest triangulated subcategory of $\underline{\textup{vect}}\mbox{-}\mathbb{X}$ containing $T$ coincides with $\underline{\textup{vect}}\mbox{-}\mathbb{X}$.
\end{itemize}

\begin{thm}\textup{\cite[Theorem.6.1]{[KLM]}}\label{tilting extension bundles}
Let $\mathbb{X}$ be a weighted projective line of weight type $(p_1,p_2,p_3)$. Then  $T_{\textup{cub}}=\bigoplus\nolimits_{0\leq\vec{x}\leq2\vec{\omega}+\vec{c}}E\langle\vec{x}\rangle$ is a tilting object in $\underline{\textup{vect}}\mbox{-}\mathbb{X}$.
\end{thm}

The main result of this section is to give a positive answer to Question \ref{bB}.

From now onwards, let $\mathbb{X}$ be a weighted projective line of weight type $(2,p,q)$ with $p,q\geq 2$.
In this case, we have

\begin{lem}\label{shift functor is given by line bundle twist of x1}\textup{\cite[Proposition 6.8]{[KLM]}}
For weight type $(2, p, q)$ with $p,q\geq 2$, the suspension functor on $\underline{\textup{vect}}\mbox{-}\mathbb{X}$ is given by the line bundle twist $X\to X(\vec{x}_1)$.
\end{lem}

Let $E$ be an Auslander bundle.
Denote by $$T_1=\bigoplus\{E(a\bar{x}_2+b\bar{x}_3)|0\leq a\leq q-2, 0\leq b\leq p-2\}.$$

\begin{prop}\label{3.7} For any $\vec{x}\in\mathbb{L}$,
 $E(\vec{x})\in\langle T_1\rangle$.
\end{prop}

\begin{proof}
Write $\vec{x}$  in normal form as $\vec{x}=\sum\limits_{i=1}^3 l_i\vec{x}_i+l\vec{c}$. Then by Lemma \ref{shift functor is given by line bundle twist of x1},
$$E(\vec{x})=E(\sum\limits_{i=1}^3l_i\vec{x}_i+l\vec{c})=E(l_2\vec{x}_2+l_3\vec{x}_3)[l_1+2l].$$
Since $\langle T_1\rangle$ is closed under the suspension functor action, it suffices to show that
\begin{equation}\label{only need to proof in 3.8}
E(l_2\vec{x}_2+l_3\vec{x}_3)\in\langle T_1\rangle \textup{\ \ for\ \ }0\leq l_2\leq p-1,\ 0\leq l_3\leq q-1.
\end{equation}
For any $0\leq a\leq q-2$ and $0\leq b\leq p-2$,
$$a\bar{x}_2+b\bar{x}_3=a(\vec{x}_2+\vec{\omega})+b(\vec{x}_3+\vec{\omega})=(a+b)\vec{x}_1+(p-b)\vec{x}_2+(q-a)\vec{x}_3-2\vec{c},$$
hence
\begin{equation}\label{cite 1}\begin{array}{ll}
E((p-b)\vec{x}_2+(q-a)\vec{x}_3)=E(a\bar{x}_2+b\bar{x}_3)[4-(a+b)]\in\langle T_1\rangle.
\end{array}
\end{equation}
Therefore, (\ref{only need to proof in 3.8}) holds except the cases $l_2=1$ or $l_3= 1$.

First we show that $E(\vec{x}_2)\in\langle T_1\rangle$. By (\ref{important triangles}) we have a sequence of triangles $\eta_k$ for $1\leq k\leq p-2$ as follows:
\begin{equation}\label{triangle series}
\eta_k: E\langle (k-1)\vec{x}_2 \rangle \to E\langle k\vec{x}_2 \rangle \to E(k\vec{x}_2)\to E\langle (k-1)\vec{x}_2\rangle[1].
\end{equation}
By (\ref{cite 1}) we know that $E(k\vec{x}_2)\in\langle T_1\rangle$ for any $k\neq 1$, then by (\ref{extension bundle shift to be auslander bundle}) we have $$E\langle(p-2)\vec{x}_2\rangle=E((p-1)\vec{x}_2)[-1]\in\langle T_1\rangle.$$
By recursively analysing on the triangles $\{\eta_{k}|2\leq k\leq p-2\}$ from $k=p-2$ to $2$, we obtain that $$E\langle k\vec{x}_2\rangle\in\langle T_1\rangle \text{\quad for\quad} 1\leq k\leq p-3.$$ Then both of $E$ and $E\langle \vec{x}_2 \rangle$ in $\eta_1$ belong to $\langle T_1\rangle$, which implies that $E(\vec{x}_2)\in\langle T_1\rangle$.

Secondly, we claim that $E(\vec{x}_2+j\vec{x}_3)\in\langle T_1\rangle$ for $2\leq j\leq q-1$. In fact, consider the following triangles $\eta_k(j\vec{x}_3)$ obtained from (\ref{triangle series}) by twisting with $j\vec{x}_3$:
$$E\langle (k-1)\vec{x}_2 \rangle(j\vec{x}_3) \to E\langle k\vec{x}_2 \rangle(j\vec{x}_3) \to E(k\vec{x}_2+j\vec{x}_3)\to E\langle (k-1)\vec{x}_2\rangle(j\vec{x}_3)[1].$$
Note that for $j\neq 1$ and $k\neq 1$,
$$E(k\vec{x}_2+j\vec{x}_3)\in\langle T_1\rangle \text{\;\;and \;\;} E\langle(p-2)\vec{x}_2\rangle(j\vec{x}_3)=E((p-1)\vec{x}_2+j\vec{x}_3)[-1]\in\langle T_1\rangle.$$
By recursively analysing on the triangles $\{\eta_{k}(j\vec{x}_3)|2\leq k\leq p-2\}$ from $k=p-2$ to $2$,
we obtain
$E(\vec{x}_2+j\vec{x}_3)\in\langle T_1\rangle \textup{\ \ for\ \ } 2\leq j\leq q-1.$ This proves the claim.

Therefore, we prove that (\ref{only need to proof in 3.8}) holds when $l_2=1$ and $l_3\neq 1$. Dually, one can prove that (\ref{only need to proof in 3.8}) also holds when $l_3=1$ and $l_2\neq 1$.

Finally, we need to show that $E(\vec{x}_2+\vec{x}_3)\in\langle T_1\rangle$.
Consider the triangles $\eta_k(\vec{x}_3)$:
$$E\langle (k-1)\vec{x}_2 \rangle(\vec{x}_3) \to E\langle k\vec{x}_2 \rangle(\vec{x}_3) \to E(k\vec{x}_2+\vec{x}_3)\to E\langle (k-1)\vec{x}_2\rangle(\vec{x}_3)[1].$$
Notice that $E\langle(p-2)\vec{x}_2\rangle(\vec{x}_3)=E((p-1)\vec{x}_2+\vec{x}_3)[-1]\in\langle T_1\rangle$, and $E(k\vec{x}_2+\vec{x}_3)\in\langle T_1\rangle$ for $2\leq k\leq p-2$. It follows that $E\langle k\vec{x}_2\rangle(\vec{x}_3)\in\langle T_1\rangle$ for $1\leq k\leq p-3$ by recursively analysing on the triangles $\{\eta_{k}(\vec{x}_3)|2\leq k\leq p-2\}$ from $k=p-2$ to $2$. Then in the triangle
$$\eta_1(\vec{x}_3): E(\vec{x}_3)\to E\langle \vec{x}_2 \rangle(\vec{x}_3) \to E(\vec{x}_2+\vec{x}_3)\to E(\vec{x}_3)[1],$$
both of $E(\vec{x}_3)$ and $E\langle \vec{x}_2 \rangle(\vec{x}_3)$ belong to $\langle T_1\rangle$. Hence $E(\vec{x}_2+\vec{x}_3)\in\langle T_1\rangle$.

This finishes the proof.
\end{proof}

Now we give the main result of this subsection.

\begin{thm}\label{tilting auslander bundle 1}
\begin{itemize}
\item[(1)] $T_1$ is a tilting object in $\underline{\textup{vect}}\mbox{-}\mathbb{X}$;

\item[(2)] the endomorphism algebra $\underline{\textup{End}}(T_1)^{op}\cong {\bf{k}}Q_1/I_1$, where $Q_1$ has the following shape and $I=\langle xy-yx,x^2,y^2\rangle$:
$$\xymatrix{
  (0,0) \ar[d]^{x} \ar[r]^{y} & (0,1) \ar[d]^{x} \ar[r]^{y} & (0,2) \ar[d]^{x}
  \ar@{.}[r]^{} & (0, p-2) \ar[d]^{x} \\
 (1,0) \ar[d]^{x} \ar[r]^{y} & (1,1) \ar[d]^{x} \ar[r]^{y} & (1,2) \ar[d]^{x}
  \ar@{.}[r]^{} & (1, p-2) \ar[d]^{x} \\
 (2,0)  \ar@{.}[d]\ar[r]^{y} & (2,1)  \ar@{.}[d] \ar[r]^{y} & (2,2)  \ar@{.}[d]
  \ar@{.}[r]^{} & (2, p-2)  \ar@{.}[d]\\
  (q-2,0) \ar[r]^{y} & (q-2,1)  \ar[r]^{y} & (q-2,2)
  \ar@{.}[r]^{} & (q-2, p-2).}$$
\end{itemize}
\end{thm}

\begin{proof}
First, we show that $T_1$ is extension-free, i.e., $\underline{\textup{Hom}}(T,T[n])=0$ for any $n\neq 0$.

For contradiction assume  $\underline{\textup{Hom}}(E(a\bar{x}_2+b\bar{x}_3), E(a'\bar{x}_2+b'\bar{x}_3)[n])\neq 0$ for some $0\leq a, a' \leq q-2 $ and $0\leq b, b'\leq p-2$. By Lemma \ref{shift functor is given by line bundle twist of x1} we have
$$\underline{\textup{Hom}}(E(a\bar{x}_2+b\bar{x}_3), E(a'\bar{x}_2+b'\bar{x}_3)[n])\cong\underline{\textup{Hom}}(E, E((a'-a)\bar{x}_2+(b'-b)\bar{x}_3+n\vec{x}_1)).$$
Denote by $\vec{y}=(a'-a)\bar{x}_2+(b'-b)\bar{x}_3+n\vec{x}_1$. Then $\vec{y}$ has the following expression:
\begin{equation}\label{expression of y}
\vec{y}=(a'-a+b'-b+n)\vec{x}_1+(b-b')\vec{x}_2+(a-a')\vec{x}_3.
\end{equation}
By Lemma \ref{BasicPropInKLM}(1), $\vec{y}=0$ or $\bar{x}_i$ for some $1\leq i\leq 3$.
\begin{itemize}
\item[(i)]  $\vec{y}=0$, then we get $p|b-b', q|a-a'$, which imply that $a=a', b=b'$. It follows that $n\vec{x}_1=0$. Hence $n=0$.
\item[(ii)]  $\vec{y}=\bar{x}_1=\vec{c}-\vec{x}_2-\vec{x}_3$, then comparing the coefficients of $\vec{x}_2$ and $\vec{x}_3$ with (\ref{expression of y}), we get $b-b'=-1$, $a-a'=-1$, and $a'-a+b'-b+n=2$.
It follows that $n=0$.
\item[(iii)]  $\vec{y}=\bar{x}_2=\vec{x}_1-\vec{x}_3$, then $b-b'=0$, $a-a'=-1$, and $a'-a+b'-b+n=1$.
Hence $n=0$.
\item[(iv)] $\vec{y}=\bar{x}_3=\vec{c}-\vec{x}_1-\vec{x}_2$, then by similar arguments as in (iii), we obtain $b-b'=-1, a-a'=0$ and then $n=0$.
\end{itemize}

Secondly, we show that $T_1$ generates the triangulated category $\underline{\textup{vect}}\mbox{-}\mathbb{X}$. In fact, by Proposition \ref{3.7}, $E(\vec{x})\in\langle T_1\rangle$ for any $\vec{x}\in\mathbb{L}$. By Theorem \ref{tilting extension bundles}, $T_{\textup{cub}}$ is a tilting object in $\underline{\textup{vect}}\mbox{-}\mathbb{X}$, and by \cite[Corollary 4.22]{[KLM]}, $T_{\textup{cub}}\in\langle E(\vec{x})\,|\,0\leq\vec{x}\leq2\vec{\omega}+\vec{c}\rangle \subseteq\langle T_1\rangle$. Hence $\underline{\textup{vect}}\mbox{-}\mathbb{X}=\langle T_{\textup{cub}}\rangle=\langle T_1\rangle$.

Finally, by Lemma \ref{BasicPropInKLM} (1),  $\underline{\textup{Hom}}(E(a\bar{x}_2+\bar{x}_3), E(a'\bar{x}_2+b'\bar{x}_3))\neq 0$ if and only if $(a',b')=(a,b),(a+1,b),(a,b+1)$ or $(a+1,b+1)$. Moreover, in these cases, the Hom-space has dimension one. Hence the endomorphism algebra $\underline{\textup{End}}(T_1)^{op}\cong{\bf{k}}Q_1/I_1$, where each point $(a,b)$ in $Q$ corresponds to the Auslander bundle $E(a\bar{x}_2+b\bar{x}_3)$ for $0\leq a\leq q-2, 0\leq b\leq p-2$, and $x$ and $y$ correspond to the basis of the one dimensional space $\underline{\textup{Hom}}(E, E(\bar{x}_2))$ and $\underline{\textup{Hom}}(E, E(\bar{x}_3))$ respectively, and $I_1=\langle xy-yx,x^2,y^2\rangle$.

This finishes the proof.
\end{proof}

As an application of Theorem \ref{tilting auslander bundle 1}, we extend the result of \cite[Theorem 4.5]{Lenzing2011Weighted} to the more general case as follows.

Denote by
$$T_2=\bigoplus\{E(a\bar{x}_1+b\bar{x}_3)|0\leq a\leq q-2,0\leq b\leq p-2\}.$$

\begin{thm}\label{tilting auslander bundle 2}
\begin{itemize}
\item[(1)] $T_2$ is a tilting object in $\underline{\textup{vect}}\mbox{-}\mathbb{X}$;
\item[(2)] the endomorphism algebra $\underline{\textup{End}}(T_2)^{op}\cong {\bf{k}}Q_2/I_2$, where $Q_2$ has the following shape
and $I_2=\langle xy-yx,x^2, y^2\rangle$:
\begin{align*}\xymatrix{
  (0,0)  \ar[r]^{y} & (0,1) \ar[ld]_{x} \ar[r]^{y} & (0,2) \ar[ld]_{x}
  \ar@{.}[r]^{} & (0, p-3)\ar[r]^{y}\ar@{.}[ld] & (0, p-2) \ar[ld]_{x} \\
 (1,0)  \ar[r]^{y} & (1,1) \ar[ld]_{x} \ar[r]^{y} & (1,2) \ar[ld]_{x}
  \ar@{.}[r]^{} & (1, p-3)\ar[r]^{y}\ar@{.}[ld] & (1, p-2)  \ar[ld]_{x} \\
 (2,0) \ar[r]^{y} & (2,1)  \ar@{.}[ld] \ar[r]^{y} & (2,2) \ar@{.}[ld]
  \ar@{.}[r]^{} & (2, p-3)\ar[r]^{y}\ar@{.}[ld]& (2, p-2) \ar@{.}[ld]\\
  (q-2,0) \ar[r]^{y} & (q-2,1)  \ar[r]^{y} & (q-2,2)
  \ar@{.}[r]^{} & (q-2, p-3)\ar[r]^{y} & (q-2, p-2).}
  \end{align*}

\end{itemize}
\end{thm}

\begin{proof}
By similar arguments as in the first part of the proof of Theorem \ref{tilting auslander bundle 1} we obtain that $T_2$ is extension-free, and $\underline{\textup{End}}(T_2)^{op}={\bf{k}}Q_2/I_2$ where each point $(a,b)$ in $Q_2$ corresponds to the Auslander bundle $E(a\bar{x}_1+b\bar{x}_3)$ for $0\leq a\leq q-2, 0\leq b\leq p-2$, $x$ and $y$ correspond to the basis of the one dimensional space $\underline{\textup{Hom}}(E, E(\bar{x}_2))$ and $\underline{\textup{Hom}}(E, E(\bar{x}_3))$ respectively, and $I_2=\langle xy-yx, x^2, y^2\rangle$. So we only need to prove that $T_2$ generates $\underline{\textup{vect}}\mbox{-}\mathbb{X}$.

By theorem \ref{tilting auslander bundle 1}, $T_1=\bigoplus\{E(a\bar{x}_2+b\bar{x}_3)|0\leq a\leq q-2,0\leq b\leq p-2\}$ is a tilting object in $\underline{\textup{vect}}\mbox{-}\mathbb{X}$. Hence it suffices to to show that
$$E(a\bar{x}_2+b\bar{x}_3)\in\langle T_2\rangle \textup{\ \ for\ \ } 0\leq a\leq q-2 \textup{\ \ and\ \ } 0\leq b\leq p-2.$$
In fact, since $\bar{x}_1=\bar{x}_2+\bar{x}_3$, we have $$E(a\bar{x}_2+b\bar{x}_3)=E(a\bar{x}_1+(b-a)\bar{x}_3).$$
Note that $p\bar{x}_3=p(\vec{\omega}+\vec{x}_3)=p\vec{x}_1-p\vec{x}_2=p\vec{x}_1-\vec{c}=(p-2)\vec{x}_1$.
Assume $b-a=kp+\lambda$ for some integers $k$ and $\lambda$ with $0\leq \lambda\leq p-1$. Then by Lemma \ref{shift functor is given by line bundle twist of x1}, $$E(a\bar{x}_2+b\bar{x}_3)=E(a\bar{x}_1+(kp+\lambda)\bar{x}_3)=E(a\bar{x}_1+\lambda\bar{x}_3)[k(p-2)].$$
Hence it suffices to show that
$$E(a\bar{x}_1+\lambda\bar{x}_3)\in\langle T_2\rangle \textup{\ \ for\ \ }  0\leq a\leq q-2 \textup{\ \ and\ \ }
0\leq \lambda\leq p-1.$$
If $0\leq \lambda\leq p-2$, then $E(a\bar{x}_1+\lambda\bar{x}_3)$ is a direct summand of $T_2$, which of course belongs to $\langle T_2\rangle$. So we only need to show that
$$E(a\bar{x}_1+(p-1)\bar{x}_3)\in\langle T_2\rangle \textup{\ \ for\ \ }  0\leq a\leq q-2.$$

For convenience, we denote by $\vec{z}:=a\bar{x}_1+(p-1)\bar{x}_3$ in the following and claim that $E(\vec{z})\in\langle T_2\rangle$.
Note that
$\vec{x}_2=\vec{x}_1-\bar{x}_3$. For $1\leq k\leq p-1$, we have
$$E( k\vec{x}_2)(\vec{z})=E(k\vec{x}_2+a\bar{x}_1+(p-1)\bar{x}_3)=E(a\bar{x}_1+(p-1-k)\bar{x}_3)[k]\in\langle T_2\rangle.$$
It follows that
$$E\langle (p-2)\vec{x}_2\rangle(\vec{z})=E((p-1)\vec{x}_2+\vec{z})[-1] \in\langle T_2\rangle.$$
Consider the triangles $\eta_{k}(\vec{z})$ obtained from (\ref{triangle series}) by twisting with $\vec{z}$. By recursively analysing on the triangles $\{\eta_{k}(\vec{z})|1\leq k\leq p-2\}$ from $k=p-2$ to $1$, we finally obtain $E(\vec{z})\in\langle T'\rangle,$ as claimed. This finishes the proof.
\end{proof}

Combining with Theorem \ref{tilting auslander bundle 1}, Theorem \ref{tilting auslander bundle 2} and \cite[Theorem 6.1]{[KLM]}, we obtain derived equivalences between certain algebras:
\begin{cor}
Let $\mathbb{X}$ be a weighted projective line of weight type $(2, p, q)$ with $p,q\geq 2$. The following endomorphism algebras are derived equivalent.
\begin{itemize}
\item[-] $\textup{End}(T_{\textup{cub}})^{op}$ in Theorem \ref{tilting extension bundles};
\item[-] $\underline{\textup{End}}(T_1)^{op}$ in Theorem \ref{tilting auslander bundle 1};
\item[-] $\underline{\textup{End}}(T_2)^{op}$ in Theorem \ref{tilting auslander bundle 2}.
\end{itemize}
\end{cor}

\noindent {\bf Acknowledgements.}\quad
This work was supported by the Natural Science Foundation of Xiamen (No. 3502Z20227184), the Natural Science Foundation of Fujian Province (No. 2022J01034), the National Natural Science Foundation of China (No.s 12271448 and 12301054), and the Fundamental Research Funds for Central Universities of China (No. 20720220043).

\bibliographystyle{plain}

\end{document}